\documentclass[10pt]{amsart}
\usepackage{geometry}     
\usepackage{amssymb}              
\usepackage{graphicx}
\usepackage{color}

\newtheorem{theorem}{Theorem}[section]

\newtheorem{corollary}[theorem]{Corollary}

\theoremstyle{definition}
\newtheorem{definition}[theorem]{Definition}

\newtheorem{question}[theorem]{Question}
\newtheorem{proposition}[theorem]{Proposition}
\newtheorem{conjecture}[theorem]{Conjecture}

\theoremstyle{remark}

\newcommand{\Z}{\mathbb{Z}}

\newcommand{\R}{\mathbb{R}}

\newcommand{\spinc}{\operatorname{Spin}^c}
\newcommand{\arf}{\operatorname{Arf}}
\newcommand{\os}{Ozsv\'ath and Szab\'o}

\numberwithin{equation}{section}

\begin{document}

\title[Non-coherent band surgery and site-specific recombination]{Recent advances on the non-coherent band surgery model for site-specific recombination}
\author{Allison H. Moore}
\address{Department of Mathematics, University of California, Davis, One Shields Avenue, Davis, CA 95616}

\email{amoore@math.ucdavis.edu}
\author{Mariel Vazquez}
\address{Department of Microbiology and Molecular Genetics / Department of Mathematics, University of California, Davis, One Shields Avenue, Davis, CA 95616}

\email{mariel@math.ucdavis.edu}

\thanks{AM and MV were partially supported by DMS-1716987.  MV was also partially supported by CAREER Grant DMS-1519375 and DMS-1817156. }                                       


\date{}

\begin{abstract} 
Site-specific recombination is an enzymatic process where two sites of precise sequence and orientation along a circle come together, are cleaved, and the ends are recombined. Site-specific recombination on a knotted substrate produces another knot or a two-component link depending on the relative orientation of the sites prior to recombination. Mathematically, site-specific recombination is modeled as coherent (knot to link) or non-coherent (knot to knot) banding. We here survey recent developments in the study of non-coherent bandings on knots and discuss biological implications.
\end{abstract}


\maketitle

\section{Introduction}
\label{sec:intro}

Local events where two linear segments along a curve in three-dimensional space come together and undergo cleavage followed by reconnection have been observed in a surprising variety of natural settings and at widely different scales. Examples range from microscopic site-specific recombination on circular DNA molecules, to reconnection of fluid vortices and magnetic reconnection in solar coronal loops \cite{Li2016, Kleckner2013aa, Shimokawa}. 

Enzymes in the family of site-specific recombinases mediate local reconnection on DNA molecules in order to address a variety of biological problems. We refer to the DNA molecule as a polymer chain, or simply a \emph{chain}. Because local reconnection reactions introduce two breaks and rearrange the strands of the chains involved, when acting on circular chains they can induce topological changes. In particular, they can resolve problematic knotting and linking that threatens the health of the cell. Experimental data show striking similarities between the pathways of topology simplification of newly replicated circular DNA plasmids \cite{Shimokawa, Stolz2017} by recombination and those of interlinked fluid vortices \cite{Kleckner, KKI, Laing}, thus pointing to a universal process of unlinking by local reconnection. 
This process has been the focus of numerous experimental and theoretical studies, e.g. \cite{Ip2003, Grainge2007, Shimokawa, BuckIshihara2015, BIRS, Stolz2017}.

Here, we are interested in \emph{site-specific recombination} on circular DNA molecules, where site-specific recombinases perform local reconnection on two short DNA segments with identical nucleotide sequences, the \emph{recombination sites} (reviewed in Section \ref{sec:ssr}). We model knotted and linked DNA molecules as topological knots and links in the three-sphere, and the recombination sites as two oriented segments on the knot or link (one in each component). Two recombination sites on a DNA knot are in direct repeat or in inverted repeat depending on whether their DNA sequences induce the same or opposite orientations into the knot, respectively (Figure \ref{fig:siteor}). 

Site-specific recombination on a knot is modeled by \emph{coherent} band surgery if the recombination sites are in direct repeats, and as \emph{non-coherent} band surgery if the sites are in inverted repeats. Non-coherent band surgery on a knot yields another knot, while coherent band surgery on a knot yields a link of two components. Band surgery along knots and links is defined in Section \ref{sec:bandsurgery}. Both operations have biological relevance. For example in \emph{Escherichia coli}, site-specific recombination is used to monomerize dimers that arise as products of homologous recombination on damaged replication links (reviewed in \cite{LeBourg}). Mathematically, in its simplest form, this is a coherent band surgery taking an unknot to an unlink. However, monomerizing DNA dimers can also refer to splitting a knot of length $l$ into a two-component link where each component is of length $l\over 2$. Coherent banding is also used by enzymes to integrate viral DNA into the genomes of their hosts, and by transposons, in a two-step reaction, to move pieces of DNA from one region of the genome to another. Non-coherent banding is observed in nature in processes that require inversion of a DNA segment. For example, recombinases Gin of bacteriophage Mu and Hin of \emph{Salmonella} recognize two specific sites in inverted repeats along a DNA molecule, and invert the segment of DNA between the sites. The action of these enzymes is reviewed in section \ref{sec:enzymes}. 

While the coherent band surgery model for site-specific recombination has received much attention in the literature, (e.g. \cite{ISV, BuckIshihara2015, BIRS, DIMS, Shimokawa}), non-coherent band surgery is less-understood and is intrinsically difficult from an analytical perspective. This is due to the implicit existence of a non-orientable surface induced by the banding. In Section \ref{sec:bandsurgery} we give an overview of non-coherent band surgery and mention several topological criteria for the existence of a banding relating a pair of knots. 

Our own approach to the study of non-coherent band surgery and its applications to site-specific recombination is two-fold: analytical, described in Section \ref{sec:new}; and numerical, described in Section \ref{sec:numerical}. Analytically, a band surgery operation relating a pair of knots is understood as a certain tangle replacement which may be ``lifted" to rational Dehn surgery in the branched double covers of the knots involved. The theory of Dehn surgery and knot invariants, together with some heavy machinery from the realm of Floer homology can then be used to study such operations. The goal in this approach is to produce theoretical results that can be used to obstruct, classify, or characterize band surgery operations. We refer the reader to Section \ref{sec:new}. 

The second facet of our approach, described in Section \ref{sec:numerical}, is computational. A band surgery operation relating a pair of knots can be realized as an operation on knotted polygonal chains embedded in three-dimensional space. Here we use BFACF, a well-known Markov Chain Monte Carlo algorithm, to uniformly sample random conformations of knotted self-avoiding polygons (SAPs) in the simple cubic lattice $\Z^3$  (reviewed in \cite{MadrasSlade}). Using \emph{Recombo}, our own software package, we algorithmically identify binding sites, perform recombination and identify the resulting product knots \cite{Stolz2017}. A robust statistical analysis of the resulting transition probability networks gives us insight into the quantitative behavior of topological simplification via recombination.

In sum, we here focus on the non-coherent band surgery model for site-specific recombination on circular DNA molecules at sites in inverted repeat. In Section \ref{sec:ssr}, we provide biological background, review site-specific recombination on circular DNA, and provide examples of these events in the biological literature. In Section \ref{sec:bandsurgery}, we define band surgery, and survey the current state of non-coherent band surgery results in low dimensional topology. In Section \ref{sec:new} we discuss recent progress by the authors in non-coherent band surgery resulting from related work in three-manifolds. In Section \ref{sec:numerical}, we give an overview of our recent numerical studies of non-coherent band surgery along knots in the simple cubic lattice.

\section{DNA recombination}
\label{sec:ssr}

\begin{figure}
\includegraphics[width=0.4\textwidth]{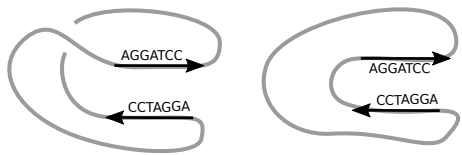}
\caption{Each figure illustrates a circular DNA molecule with two recombination sites. The DNA is modeled as the curve drawn by the axis of the double helix (see Section \ref{sec:ssr}). The two sites have the same nucleotide sequence, AGGATCC. Since the sequence is non-palindromic, one can assign an orientation to it. (Left) A pair of sites in inverted repeats, i.e. the sites induce opposite orientations into the circle (head-to-head). (Right) A pair of sites in direct repeats (head-to-tail).}
\label{fig:siteor}
\end{figure}

\subsection{DNA, damage and recombination.} 
The nucleic acids DNA and RNA are molecules that carry the genetic code of every organism. DNA is composed of two sugar-phosphate backbones flanked by nitrogenous bases A, G, T and C (adenine, guanine, thymine and cytosine). The backbones are held together by hydrogen bonds between each pair of complementary bases (A with T; G with C). The two backbones wrap around each other, forming the familiar right-handed double-helix. Words in the nucleotide sequence determine the genetic code. The genetic code is unique to each individual and can be used to differentiate a bacterium from a human, or two humans from each other. Double-stranded DNA molecules may be linear or circular. Most bacterial organisms have circular DNA. Length is measured by the number of base pairs (bp) of nucleotides, and varies widely across different organisms. For example, the chromosome of the bacterium {\it Escherichia coli} is circular, with length ranging from 4.5 to 5.5 million bp. The human genome is comprised of 23 chromosomal pairs of linear DNA molecules, and is approximately 3 billion bp long. Circular DNA molecules may trap interesting geometrical and topological information. For example, links naturally arise as products of DNA replication on circular DNA molecules. Replication is the process by which a cell copies its genetic code in preparation for cell division. Interlinking of chromosomes prevents the newly replicated DNA molecules from properly segregating at cell division. Enzymes that mediate local changes can correct for topological problems; local changes include strand-passage and reconnection, and have the potential to change the global topological type of a knotted or linked DNA chain. The action of these enzymes often leads to topological simplification (e.g. a reduction in minimal crossing number) of the substrate DNA knots and links \cite{Rybenkov1997, Zech, Grainge2007, LeBourg}. 

When exposed to sparsely ionizing radiation, or to some chemotoxic agents such as chemotherapy drugs, the DNA in cells is subjected to double-stranded breaks (DSBs). If not repaired properly, the DNA damage may result in misrejoinings that lead to large-scale genome rearrangements (reviewed in \cite{Hlatky}). Depending on the extent of the damage and on the phase of the cell cycle when it happens, DSBs are repaired by non-homologous end-joining (NHEJ) or by homologous recombination (HR). The two repair pathways have widely different levels of accuracy, HR being the most faithful. However, in both cases the process of DNA damage and repair starts with two DNA sites in close proximity that undergo cleavage. The loose ends are ideally put back together without any DNA loss or rearrangement. Occasionally the ends are resealed with the wrong partner thus leading to a chromosomal rearrangement. If the strands of DNA are swapped during HR, the resulting local reconnection event is referred to as a {\it cross-over}. Chromosomal rearrangements induced by DNA damage may lead to chromosomal instability and cell death. Chromosomal rearrangements, in particular DNA inversions, are also considered as drivers of genome evolution \cite{Wellenreuther, Zhang}.

Reconnection events mediated by site-specific recombinases are common in nature, and have been extensively studied from the knot theoretic perspective. During a site-specific recombination event the enzymes recognize and bind two recombination sites. The nucleotide sequence along the two sites is identical. The two sites may be found along one DNA molecule or two distinct molecules. Site-specific recombinases act by cleaving both segments, and then recombining and joining the ends. Because the sites targeted by site-specific recombinases are short (typically 5-50 bp) and the mechanism is highly specific, these enzymes are often used for genetic engineering.
They also are ubiquitous in the natural environment, and are involved in the integration or excision of genetic material \cite{gottesman1971}, dimer resolution \cite{Stirling1988}, or the inversion of genetic material to alter gene expression \cite{heichmanjohnson1990}. 

A recombination site is usually a short, non-palindromic word composed of letters A, T, C, G corresponding to its nucleotide sequence. Hence a given site can be assigned an unambiguous orientation (e.g., think of a `word' as being read left-to-right). If two sites appear along a single circular chain, they will induce either the same or different orientations along the chain. In the first case they are said to be in {\it direct repeats}. In the second case, they are said to be in {\it inverted repeats}. Figure \ref{fig:siteor} shows a schematic of two sites in inverted repeats and direct repeats, respectively, along a single circular molecule. 

\begin{figure}[b]
\includegraphics[width=0.6\textwidth]{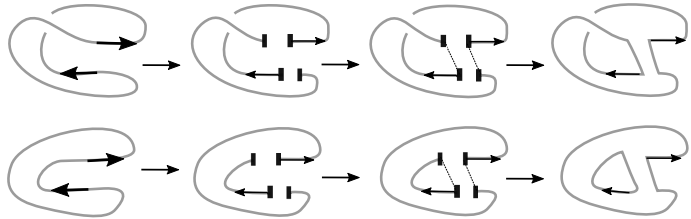}
\caption{(Top) The product of site-specific recombination at inversely repeated sites along a knot is a knot. (Bottom) The product of site-specific recombination at directly repeated sites along a knot is a two-component link.}
\label{fig:reconnection}
\end{figure}

We make the following simple combinatorial observation. Given a single component knot containing two sites in direct repeats, recombination will yield a two-component link, perhaps non-trivially linked. Given a single component knot containing two inversely repeated sites, the product will yield another single component knot, perhaps of a different knot type. See Figure \ref{fig:reconnection}. Conversely, given a two-component link with one site on each component, the product is necessarily a knot, now containing two directly repeated sites. These distinctions will become quite important when we discuss the different band surgery models in Section \ref{sec:bandsurgery}. In the biology literature, when two distinct molecules are fused into one molecule in the recombination process, the product is called a \emph{dimer}. Likewise, in biology two circular molecules linked together are called \emph{catenanes}.

\begin{figure}
\includegraphics[height=0.21\textwidth]{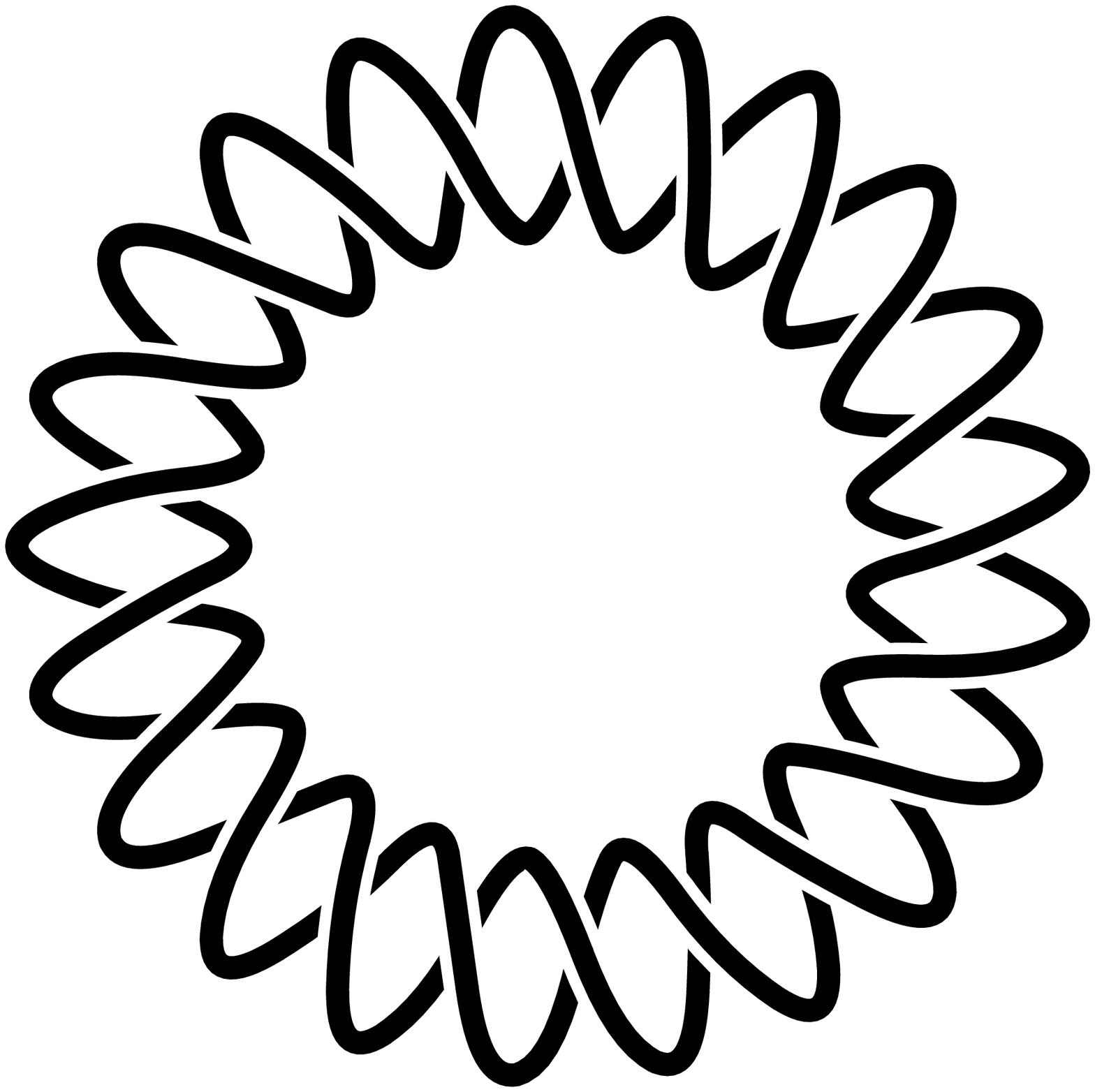} \qquad \includegraphics[height=0.18\textwidth]{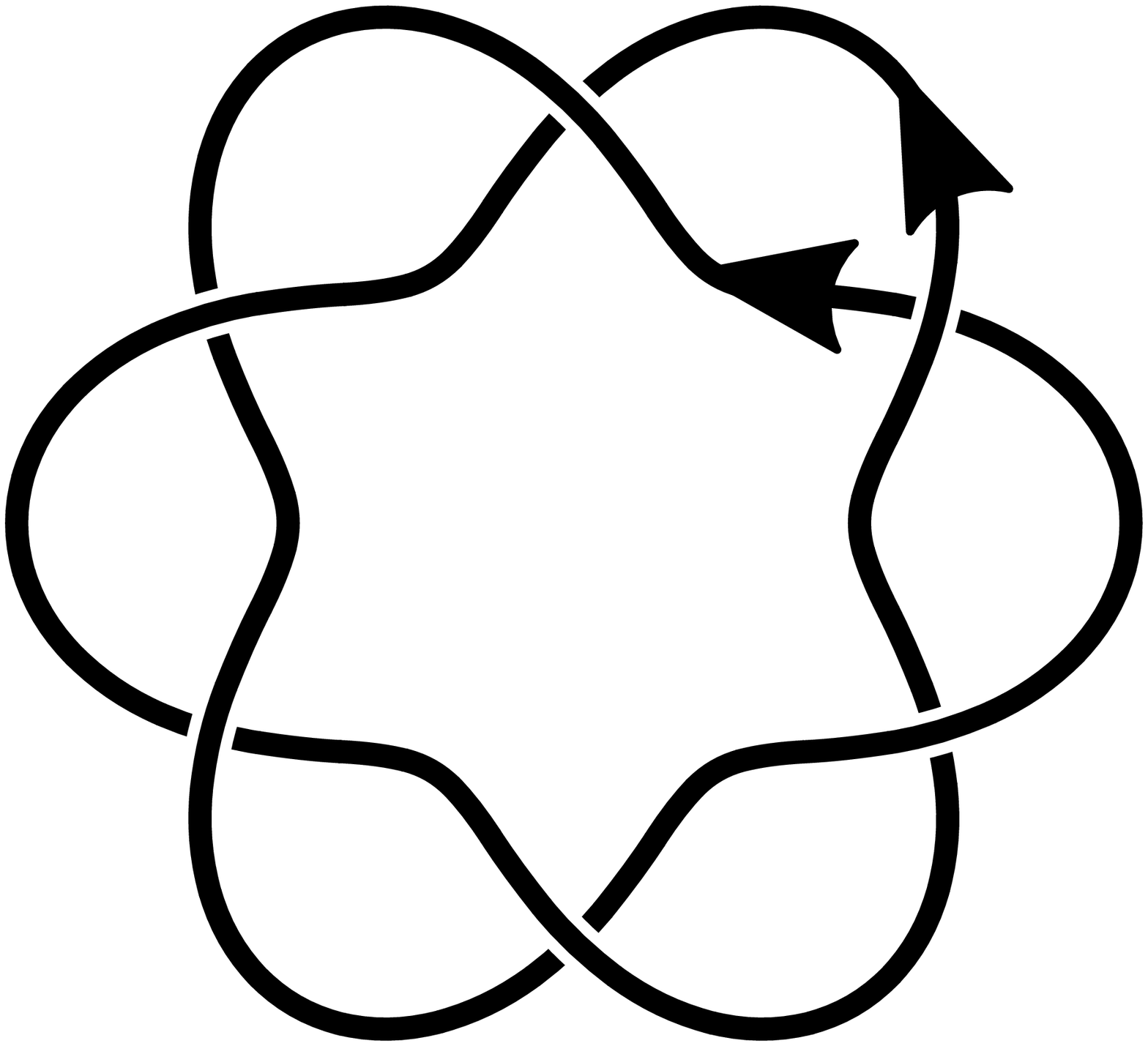} 
\caption{The image on the left shows an unoriented $T(2, 21)$ torus knot and the image on the right shows an oriented $T(2,6)=2^2_6$ torus link. The two sites indicated by arrows are said to be in parallel alignment. In the biology literature, $T(2,n)$ torus links are also called $n$-catenanes. This link illustrates one possible product of DNA replication on an unknotted substrate.}
\label{fig:torusknot}
\end{figure}

\subsection{The relevance of $T(2, n)$ torus knots and links in biology.} 
Certain classes of knots and links are especially relevant in the biological context for mechanistic reasons, for example two-bridge links and torus links. Consider the two interlinked backbones of a DNA double-helix. If the DNA is circular, the backbone naturally forms a $T(2,2n)$ right-handed torus link, where $n$ corresponds to the number of turns of the helix. See Figure \ref{fig:torusknot} for an example. During DNA replication, enzymes called helicases unwind and `unzip' the double-helix at the replication fork. The replication process yields two interlinked daughter chromosomes. The curves drawn by the axes of each DNA double helix form a two-component link that has been experimentally confirmed to be a $T(2, 2n)$ torus link \cite{AdamsCozz92}. Notice that the link type is a direct consequence of the right-handed double-helical structure of DNA. To ensure the survival of the cell line, these two components must become unlinked. DNA unlinking mediated by enzymes is reviewed in the next subsection

Knotting in long circular DNA molecules may arise as a result of enzymatic reactions, or as products of a random knotting process. The Frisch-Wasserman-Delbr\"{u}ck Conjecture \cite{FW, DM} asserts that in long polymer chains knotting will occur with near certainty. The conjecture is known to hold for various polymer models \cite{SW,D1,DPS} and has been experimentally verified on randomly circularized DNA chains \cite{Liu81, Rybenkov93, Shaw93, Arsuaga02}. Furthermore, when polymer chains are subjected to confinement, for example in a cell nucleus or viral capsid, the knotting probability and knot complexity appears to increase (e.g. \cite{Arsuaga02}). The most probable knot observed in random knotting processes is the trefoil knot $T(2, 3)$.

\subsection{Enzymes that change the topology of DNA}
\label{sec:enzymes}

This is a good time to interrupt the exposition in order to clarify some biological naming conventions for the mathematical reader. General classes of enzymes are usually suffixed by `-ase' (e.g. helicase, topoisomerase, recombinase). The name of a particular enzyme is capitalized (e.g. Gin, Hin, Xer, or $\lambda$-Int), and the name of a recombination site is indicated in italics (e.g. \emph{att, dif, psi, res}). 

In E. coli, topological simplification of replication links is mediated by type II topoisomerases. These are enzymes that cut one segment of double-stranded DNA, pass another duplex through the break, and then reseal the broken ends. Strand-passage mediated by type II topoisomerases is typically modeled by crossing changes on curves, linear or circular, representing the DNA.
Type II topoisomerases are ubiquitous and essential to life. They are believed to be the predominant decatenases in the cell (i.e. enzymes in charge of unlinking interlinked DNA molecules) \cite{Zech}. However, in a series of \emph{in vitro} experiments Ip {\it et al.} showed that site-specific recombinases can also mediate topological simplification \cite{Ip2003}. In particular, they demonstrated that the recombinases XerC/XerD, acting at \emph{dif}-sites, and coupled with an auxiliary protein FtsK, were capable of resolving plasmid substrates that had been linked by $\lambda$-Int, a site-specific recombinase from bacteriophage $\lambda$ \cite{BSC}. The same research group later showed that the XerCD-\emph{dif}-FtsK enzymatic complex could unlink replication links tied \emph{in vivo} in E. coli cells deficient in the decatenase topoIV \cite{Grainge2007}. Several groups of authors, including our own, studied mathematically the mechanism and pathways of DNA unlinking by site-specific recombination \cite{BuckIshihara2015, BIRS, Shimokawa, Stolz2017}. In \cite{Shimokawa} we showed that at least $2m$ steps are needed to unlink a replication link of type $T(2,2n)$ with one recombination site in each component in the orientation indicated in Figure \ref{fig:torusknot}. All the intermediate knots and links are also in the $T(2,n)$ family. Also, under the assumption that each step strictly introduces the minimal crossing number of the substrate, we concluded that there is a unique shortest pathway of unlinking a $T(2,2n)$ replication link. This minimal pathway corresponds to the one proposed by experimentalists \cite{Grainge2007, Ip2003}. In \cite{Stolz2017} we extended the mathematical model of \cite{Shimokawa} and also analyzed recombination pathways numerically. We identified minimal pathways of unlinking under a variety of biologically-motivated assumptions and used the transition probabilities obtained numerically to identify a most probable minimal pathway of DNA unlinking of replication links. In \cite{BuckIshihara2015}, the authors characterize coherent band pathways between knots and two-component links of small crossing number, and in \cite{BIRS}, they characterize band surgeries between fibered links in terms of arcs on the fiber surfaces of such links. The above-mentioned studies were all concerned with site-specific recombination modeled as coherent band surgery, and point to the $T(2, n)$ torus knots and links as having heightened importance in pathways of unlinking of replication links.

In the biological literature there are numerous other studies of site-specific recombinases acting at two sites in direct repeats along a single circle, or at two sites on separate circles. Several instances specific to reactions involving the trefoil knot and other $T(2, n)$ torus knots and links are mentioned in \cite[Section 5]{LMV}. Here, we review several more specific examples of site-specific recombination on DNA knots at sites in inverted repeats. Recall that in this case, the product of recombination is also a single component knot.

An \emph{inversion} is a chromosomal rearrangement in which a segment of a DNA molecule is reversed, so that the corresponding portion of the nucleotide sequence is read backwards. In addition to potential topological changes (i.e. a change in the isotopy type of a DNA knot), site-specific recombination at sites in inverted repeat will induce an inversion. In some organisms, inversions are potentially harmful genetic abnormalities. However, in some cases, inversions may present as advantageous mutations. They are also considered to be the most common rearrangement in evolution. 

Gin and Hin are examples of enzymes in the family of serine site-specific recombinases that mediate the inversion of a DNA segment by acting at sites in inverted repeats. Gin is a site-specific recombinase from bacteriophage Mu. Bacteriophages are viruses that infect bacteria. Bacteriophage Mu infects a wide range of bacterial strains. After finding its target cell, a bacteriophage injects its DNA into the host, and the phage DNA becomes circular. Gin recombination inverts the G-segment, a 3kb portion of the viral genome flanked by the \emph{gix} recombination sites. As a result of G-segment inversion, a gene encoding for a new set of tail fibers is expressed. Tail fibers are used by bacteriophages to recognize their host. In this way Gin site specific recombination determines the family of bacteria that will be infected by the new generation of bacteriophages. Gin is known to act processively, meaning it performs recombination several times successively. During the first round of recombination, it converts an unknotted DNA circle with sites in inverted repeat into another unknot, now with an inverted G-segment. In the second round, the unknot is converted into a trefoil, and the G-segment is flipped back to its original orientation. As the reaction proceeds, the family of twist knots is generated ($3_{1}, 4_{1},5_{2},6_{1},7_{2}É$). From a topological perspective, the second step is a non-coherent banding relating an unknot to a trefoil knot, and a third step is a non-coherent banding relating the trefoil with the figure 8 knot ($4_{1}$). In Section \ref{sec:bandsurgery} we discuss this specific instance of non-coherent band surgery.

Another notable example is the Hin site-specific recombination system from \emph{Salmonella}. \emph{Salmonella} sp. are bacteria in the family Enterobacteriaceae. The cells are rod-shaped with lengths $2-5\mu m$. A few flagella used for locomotion are arranged on the cell surface. Salmonella cells express two different flagelar proteins in a process called flagelar phase variation. Expression of one protein over the other results in a change of phenotype. Diphasic agglutination was first observed in Salmonella in 1922 by F.W. Andrewes \cite{Andrewes}. Half a century later it was shown that the observed phenotypic changes were due to an inversion of a genomic region affecting flagellin genes. See \cite{Henderson} for a review of bacterial phase variation. The site-specific recombinase Hin inverts a 1kb DNA segment in the \emph{Salmonella} sp. chromosome. The invertible segment includes a promoter region for one flagellin gene and a repressor for another flagellin gene. Inversion by Hin selectively determines the expression of one gene and the repression of the other. In \cite{MerickelJohnson} experiments were conducted \emph{in vivo} and \emph{in vitro} to study the topological mechanism of action of Hin recombinase acting at \emph{hix} sites. Like Gin, when acting on unknotted DNA circles with two sites in inverted repeats, Hin recombination produces twist knots. Experiments pointed to a mechanism analogous to that of Gin. See \cite{Johnson2015} for a review of inversion by serine recombinases. Mathematical models of Gin and Hin recombination have been reported in \cite{BuckMauricio, CH-1, CH-2, VazquezSumners, VazquezThesis}.

\section{The band-surgery model for site-specific recombination}
\label{sec:bandsurgery}

\subsection{Knot theory background}
\label{knot}

A \emph{knot} $K$ is an embedding of the circle into $\R^3$ or its compactification, the three-sphere $S^3$. A \emph{link} $L$ is a disjoint union of knots, possibly linked together. A knot is a link with one component. Two links are considered to be equivalent if one can be continuously deformed into the other via an ambient isotopy of the surrounding space. We will further assume that all links are tame, meaning each component can be represented by piecewise linear curves with finitely many segments. It is common to illustrate links with link diagrams, which are planar projections containing the data of over-crossings and under-crossings. A link which is isotopic to its mirror image is called an \emph{achiral} link, otherwise it is \emph{chiral}. 
Chirality and orientation are important in biology. When distinguishing between different symmetries of links we follow the nomenclature proposed in \cite{Brasher, Witte}. Because we will mainly deal with non-coherent band surgery, unless otherwise noted, here we take knots and links as unoriented objects. 

\subsection{Band surgery}
\begin{figure}
	\includegraphics[width = 0.6\textwidth]{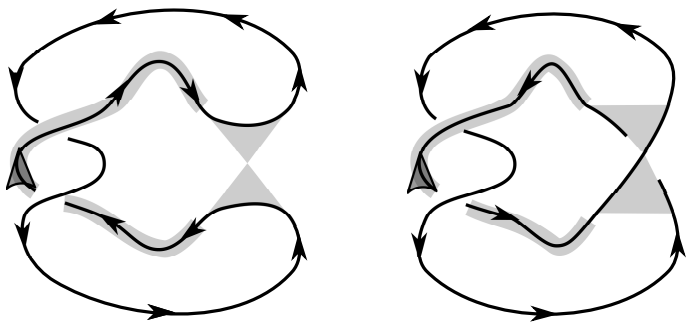}
	\caption{Notice that when a banding takes a knot to another knot, it is not possible to orient the knots in such a way that the orientations are consistent with the band surgery. Observe that the non-coherent band surgery is an inversion of the inner-most arc indicated in grey.}
	\label{orientationmismatch}
\end{figure}

Site-specific recombination on circular DNA will be modeled by a topological operation on knots and links which we now describe. Let $L$ denote a link in $S^3$. A band $b$ is an embedding of the unit square, $b(I\times I)\hookrightarrow S^3$, such that two opposing edges lie on $L$, i.e. $L\cap b(I\times I) = b(I \times \partial I)$. We say that two links $L$ and $L'$ are related by a \emph{band surgery} if 
\[
	L' = ( L - b(I\times \partial I) ) \cup  b(\partial I \times I).
\] 
In other words, given a link $L$ and a band $b$, we replace the two opposing edges of the band on $L$ with the remaining two edges to obtain a new link $L'$.

If $L$ and $L'$ are oriented links, and the orientation of $L - b(I\times\partial I)$ agrees with the specified orientations on $L$ and $L'$, then the band surgery is called \emph{coherent} (meaning it is coherent with respect to orientation). Otherwise, the band surgery is \emph{non-coherent}. 
Elsewhere in the literature, coherent band surgery is sometimes called nullification \cite{EM}, and non-coherent band surgery is called an H(2)-move, incoherent, or unoriented surgery \cite{ AbeKanenobu, Kanenobu, KanenobuMiyazawa}. A band surgery relating a pair of links induces a surface cobordism between them. A surface cobordism from $L$ to $L'$ is a surface $F$ with $\partial F = L\sqcup L'$ properly embedded in $S^3 \times I$. Band surgery induces a cobordism by attaching a single 1-handle to the surface $L \times I$ in $S^3 \times I$. In the case of coherent band surgery, the surface is oriented and the restriction to the boundary respects the given orientations on $L$ and $L'$. In the case of non-coherent band surgery, the addition of the band makes the surface non-orientable, as it contains an embedded M\"{o}bius strip. 

Let us make a similar combinatorial observation to the one we made in the case of site-specific recombination. Suppose that a band surgery along a knot yields another single component knot. This implies that the surface cobordism induced by the banding contains only two boundary components. In particular, it is non-orientable and therefore the knots cannot both be oriented in a way that is compatible with the banding. See Figure \ref{orientationmismatch}. Thus a band surgery which takes a knot to another knot is necessarily non-coherent. Similarly, a coherent band surgery operation along a knot will always produce a two-component link.  The local action of site-specific recombinases is thought of as a simple reconnection. This is modeled mathematically as a band surgery. The appropriate model for site-specific recombination is coherent band surgery when the sites are directly repeated, and is non-coherent band surgery when the sites are inversely repeated.

\begin{figure}
\includegraphics[width = 0.5\textwidth]{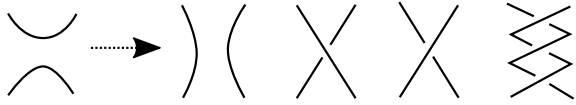}
\caption{Some of the standard tangle replacements used to model reconnection.}
\label{fig:tangle}
\end{figure}

\subsection{Band surgery via tangle replacements.} 
It is often convenient to view a band surgery between a pair of knots or links in terms of a standard diagrammatic operation. More precisely, after isotoping the links $L$ and $L'$, perhaps at the expense of complicating their embeddings, we may assume that a coherent or non-coherent band surgery is obtained from a particular two-string tangle replacement. 

A \emph{two-string tangle} is a pair $(B, t)$, where $B$ is a three-ball and $t$ is a pair of properly embedded arcs in $B$. Every link in $S^3$ can be written as the union of two complementary two-string tangles. In particular, we may describe the links $L$ and $L'$ in $S^3$ that are related by band surgery via the following two-string tangle decompositions
\begin{equation}
\label{tangledecomp}
	(S^3, L)= (B_o, t_o) \cup (B, t) \text{  and } (S^3, L')= (B_o, t_o) \cup (B, t'),
\end{equation}
where $S^3 = B_o\cup B$ and the sphere $\partial B = \partial B_o$ intersects $L$ and $L'$ transversely in four points. The pairs of properly embedded arcs are $t=(B \cap L)$, $t'=(B \cap L')$ and  $t_o = (B_o\cap L) = (B_o\cap L')$. To define a band surgery operation, the tangle $(B, t)$ is replaced with $(B, t')$. By convention, we assume that $(B, t)$ is a rational $(0)$ tangle and that $(B, t')$ is either an $(\infty)$ or $(\pm 1/n)$ tangle, as shown in Figure \ref{fig:tangle}. This convention is based on the assumption that the arcs $t$ in the tangle $(B, t)$ represent only the cleavage region of each recombination site. Since the arcs are very short DNA segments, and since DNA is a stiff biopolymer, then $(B, t)$ and $(B, t')$ can be assumed to be trivial, or very simple tangles \cite{Shimokawa, Vazquez2005}. The tangle $(B_o, t_o)$ is fixed in the operation, and is shared by both of $L$ and $L'$. Notice that subscript `o' stands for `outside' tangle, and this tangle may become arbitrarily complicated in the isotopy that simplifies the embedding of the band. In particular, $(B_o, t_o)$ need not be rational. See Figure \ref{bandisotopy} for an example.

\begin{figure}
\includegraphics[width = 0.6\textwidth]{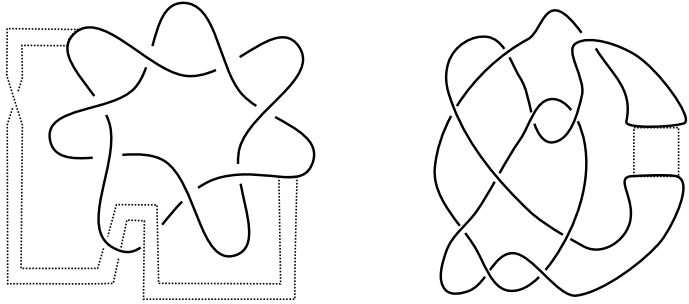}
\caption{A non-coherent band surgery, indicated by the dashed band, relates the knots $7_1 = T(2,7)$ and $5^*_2$, the mirror of $5_2$. An isotopy that simplifies the embedding of the band takes the left image to the right image. After isotopy, the outside tangle becomes complicated.}
\label{bandisotopy}
\end{figure}

\subsection{Dehn surgery and branched double covers.}
\emph{Dehn surgery} refers to a surgical operation that transforms a 3-manifold $M$ into another 3-manifold, in which an open tubuluar neighborhood of a knot or link is removed, and solid tori are glued along the boundary components \cite{Gordon}. 
More precisely, let $M$ be a compact, connected, irreducible, oriented 3-manifold with $\partial M$ a torus. A \emph{slope} is the isotopy class of an unoriented simple closed curve on the bounding torus. Given a pair of slopes $r$, $s$ the \emph{distance} $\Delta(r, s)$ refers to their minimal geometric intersection number. For any slope $r$, a closed 3-manifold $M(r)$, called a \emph{Dehn filling} of $M$, can be constructed by gluing a solid torus to $M$ by a homeomorphism which identifies a meridional curve of the solid torus to the slope $r$ on $\partial M$. For example, consider the case that $M$ is a knot complement in $S^3$. It is well known that the extended rationals $\mathbb{Q}\cup\{\infty\}$ parameterize slopes on $\partial M$. See \cite{Gordon} for a review of Dehn surgery. 

The perspective that band surgery is a two-string tangle replacement adheres to the now classical notion that the action of enzymatic complexes on circular DNA is appropriately modeled by tangle surgery, due to Ernst and Sumners in 1990 \cite{ES-1990}, reviewed in \cite{M}. When topoisomerases or recombinases act locally, they essentially split the DNA molecule into two complementary tangles, replacing one tangle with another. Viewing this operation as a topological tangle surgery opens the door to a broad range of tools in low-dimensional topology that can be used to understand these enzymatic actions. Many of the most successful topological approaches in the study of band surgery go by way of the \emph{branched double cover} of a tangle or knot complement, starting with work of Lickorish \cite{Lickorish}. Here, the branched double cover of the pair $(S^3, K)$ denoted by $\Sigma_2(S^3, K)$, is the unique closed, connected, oriented 3-manifold associated with the representation of the knot group onto $\Z/2\Z$ (c.f. \cite{Rolfsen}). Similarly, the branched double cover of a tangle $(B, t)$ is a compact, connected, oriented three-manifold with torus boundary (see,  for example \cite{BZ} or \cite[Lecture 4]{Gordon}).

The classic approach relies on combining important theorems in low-dimensional topology with the characterization of rational tangles as those with branched double covers that are solid tori. In particular, Ernst and Sumners showed that when the tangles involved are known to be rational, the Cyclic Surgery Theorem \cite{CGLS}, applied in the branched double cover, can be used to determine the types of tangles that satisfy the equations that model enzymatic actions \cite{ES-1990}.
This approach to studying the mechanism of enzymatic action is called the \emph{tangle calculus}, and it effectively reduces the problem to the straight-forward, if tedious, combinatorics of rational tangles (see \cite{ES-1990, Ernst-1, Ernst-2, ES-1999} for an extended discussion of rational tangles and the tangle method). However, there is one major limitation of tangle calculus in the sense that one must know, or assume, that the tangles involved are Montesinos tangles, i.e.  sums of rational tangles. While often biologically justified, from the topological perspective the assumption is rather restrictive. 

Another topological strategy in the study of band surgery involves analyzing surfaces bounded by the links involved. This is done in \cite{IshiharaShimokawa, BuckIshihara2015, BIRS, DIMS, Shimokawa}, for example. The essence of the approach is as follows: if a band surgery decreases the Euler characteristic of an oriented surface without closed components bounding the link, then the band can be isotoped to lie onto the Seifert surface (see for example \cite{ST} or \cite[Theorem 1.6]{HS}). If the minimizing Seifert surfaces and their genera are known, this can be used to obstruct the existence of coherent bandings, analyze nullification pathways, or characterize the tangle decompositions that yield such transitions. Here too, there are several limitations to the strategy. The first limitation is that the surface cobordism induced by the banding must be orientable, making this strategy only applicable in the case of coherent band surgery. The second limitation is that minimizing Seifert surfaces can be difficult to determine. Thus, many studies are restricted to special cases, e.g. between two-bridge knots and $T(2, n)$ torus knots \cite{DIMS} amongst fibered knots \cite{BIRS}, low crossing knots \cite{IshiharaShimokawa, BuckIshihara2015}, or make strong assumptions on topological complexity, as in \cite{Shimokawa}.

\subsection{Knot invariants and non-coherent band surgery.} 
Several groups of authors have looked at the behavior of knot invariants under non-coherent band surgery. The study of band surgery as an operation transforming a knot into other knot goes back to the 1980s when Lickorish asked whether a knot could be unknotted by adding a twisted band, and gave a criterion for this in terms of the linking form of the branched double cover \cite{Lickorish}. His approach suffices to show, for example, that the figure eight knot cannot be unknotted with a single twisted band. In fact, many of the results on non-coherent band surgery involve the branched double cover of the knot. Let $e_p$ denote the minimum number of generators of the homology group $H_1(\Sigma_p(L); \Z)$, where $\Sigma_p(L)$ is the $p$-fold branched cyclic cover. In \cite{AbeKanenobu}, the authors observed that work of Hoste, Nakanishi, and Taniyama \cite[Theorem 4]{HNT} shows the following.
\begin{proposition}\cite[Lemma 5.1]{AbeKanenobu}
\label{prop:gens}
If two knots $K$ and $K'$ are related by a single non-coherent band surgery, then $|e_p(K)-e_p(K')|\leq p-1$. 
\end{proposition}
Abe and Kanenobu also showed that certain evaluations of the Jones polynomial and Q-polynomial may be used to determine whether knots or links are related by a band surgery (either coherent or non-coherent). In the following theorem, $\omega=e^{i\pi/3}$ and $V(L; \omega)$ denotes the evaluation of the Jones polynomial $V(L; t)$ at $t^{1/2} =  e^{i\pi/6}$.
\begin{theorem}\cite[Theorem 5.5]{AbeKanenobu}
\label{thm:jones}
If two links $L$ and $L'$ are related by a single coherent or non-coherent band surgery, then
\[
	|V(L;\omega) / V(L';\omega)| \in \{ 1, \sqrt{3}^{\pm1} \}.
\]
\end{theorem}
The appearance of the Jones polynomial in Theorem \ref{thm:jones} is intrinsically related to the role of the branched double cover of $K$ in band surgery. To see this, consider the following statement of Lickorish and Millet \cite[Theorem 3]{LickorishMillett}, which states that this particular evaluation of the Jones polynomial is in fact related to the dimension of $H_1(\Sigma_2(S^3, K);\Z/3\Z)$:
\[
	V(L;\omega) = \pm i^{c-1}(i\sqrt{3})^{\delta(L)}. 
\]
Here, $c$ is the number of link components and $\delta(L)= \dim H_1(\Sigma_2(S^3, K);\Z/3\Z)$. A similar statement also holds for evaluations of the Q-polynomial at $-\bar{\phi}$, where $\bar{\phi}$ is the conjugate of the golden ratio, $\bar{\phi}=(1-\sqrt{5})/2$.
\begin{theorem}\cite[Theorem 5.5]{AbeKanenobu}
\label{thm:q}
If two links $L$ and $L'$ are related by a single coherent or non-coherent band surgery, then
\[
	|Q(L;-\bar{\phi}) / Q(L';-\bar{\phi})| \in \{ \pm 1, \sqrt{5}^{\pm1} \}.
\]
\end{theorem}
Notably, there is also a relationship between the first homology of the branched double cover with coefficients in $\Z/5\Z$ and evaluations of the Q-polynomial, due to Jones \cite{Jones}. Indeed, $Q(L;-\bar{\phi}) =\pm\sqrt{5}^r$, where $r=\dim H_1(\Sigma_2(S^3, L);\Z/5\Z)$. An additional theorem of Kanenobu \cite[Theorem 2.2]{Kanenobu} implies that if a knot or link $L$ is obtained from an unknotting number one knot $K$ by a coherent or non-coherent band surgery, then either $2\det(L)$ or $-2\det(L)$ is a quadratic residue of $\det(K)$. The proof of this statement also uses the Jones polynomial.

\subsection{Unknotting via non-coherent band surgery and crossing changes. }
Recall that the classical unknotting number $u(K)$ of a knot is defined to be the minimal number of crossing changes required to convert $K$ into an unknot. The analogously defined number in the case of non-cohernet band surgery operations is denoted by $u_2(K)$. The notation here is meant to indicate $H(2)$-unknotting number, where an $H(2)$-move is equivalent to a non-coherent band surgery. Several authors have investigated the relationship between $u(K)$ and $u_2(K)$. 

Nakajima \cite{Nakajima} has shown that unknotting number provides an upper bound on non-coherent unknotting as follows.
\begin{theorem}\cite[Theorem 3.2.3]{Nakajima}
\label{thm:bound} 
In general, $u_2(K) \leq u(K)+1$. Moreover, if both signed unknotting numbers $u_+(K)$ and $u_-(K)$ are even, then $u_2(K) \leq u(K)$.
\end{theorem}
A very clear pictorial proof of Nakajima's theorem is given in \cite[Theorem 3.1]{KanenobuMiyazawa}. 
It is worth noting that this bound is certainly not sharp. 
For example, the $T(2, q)$ torus knots have unknotting number $(q-1)/2$, but because they bound a M\"obius strip, they can be unknotted by the addition of a single band.
Kanenobu and Miyazawa have directly computed or bounded $u_2(K)$ for most knots of up to 10 crossings \cite{KanenobuMiyazawa} . Their table indicates that all knots of up to nine crossings have $u_2(K)\leq 2$, with the exception of $9_{49}$, for which $u(9_{49}) = u_2(9_{49})=3$.  

The determinant $\det(K)$, the signature $\sigma(K)$ and the Arf invariant $\arf(K)$ are integer valued invariants of knots and links defined using the Seifert form \cite{BZ}. The signature of a link is equivalently the signature of the intersection form of the four-manifold obtained by taking the double cover of the four-ball branched over a Seifert surface for the link. For knots of a single component, the determinant is odd and the signature is even. The Arf invariant takes values in $\Z/2\Z$. 
In \cite{Yasuhara}, Yasuhara gave a criterion on unknotting via banding in terms of the signature and Arf invariant.
\begin{theorem}
\label{yas}
\cite[Proposition 5.1]{Yasuhara}
If a knot $K$ bounds a M\"obius band in the four-ball, then there exists an integer $x$ such that 
\[
	| 8x + 4 \arf(K)-\sigma(K) | \leq 2.
\]
\end{theorem}
Note that slice knots are those which bound a disk smoothly embedded in the four ball, e.g. the unknot is slice. The addition of a twisted band produces a knot $K$ which bounds a M\"obius band in the four-ball. Therefore Yasuhara's criterion implies that for a knot $K$ with $u_2(K)=1$, if $\sigma(K)\equiv 0 \pmod{8}$, then $\arf(K)=0$ and if $\sigma(K)\equiv 4 \pmod{8}$, then $\arf(K)=1$.

As mentioned above, using the Jones polynomial, Kanenobu \cite[Theorem 2.2]{Kanenobu} proved that if a knot or link $L$ is obtained by a coherent or non-coherent band surgery from a knot $K$ with unknotting number one, then $2\det(L)\equiv \pm s^2 \pmod{\det(K)}$ for some integer $s$. Kanenobu and Miyazawa also gave several criteria for a knot to have non-coherent band unknotting number one via the signature, the Arf invariant, and certain evaluations of the Jones polynomial or Q-polynomial \cite{KanenobuMiyazawa}. Their work both recovers Theorem \ref{yas} and establishes the following criterion. 
\begin{theorem}
\cite[Theorem 4.5]{KanenobuMiyazawa}
Let $K$ be a knot with $u_2(K)=1$. Suppose that $\det(K)\equiv 0 \pmod{3}$ and $\sigma(K)\equiv 2\epsilon \pmod{8}$, $\epsilon =\pm1$. Then
\[
	V'(K; -1)\equiv (-1)^{\arf(K)}8\epsilon \pmod{24}
\]
\end{theorem}

The proof of Kanenobu and Miyazawa's criterion is centered on polynomial invariants, whereas the proof of Yasuhara's criterion is an application of his work on the Euler numbers of surfaces in simply connected four-manifold. 
The Jones polynomial and Q-polynomial encode topological data about abelian knot invariants in a mysterious way. Indeed, the Jones polynomial may be defined using methods from many different areas of mathematical physics. The full relationship between the Jones polynomial, Q-polynomial, and other related knot polynomials with traditional geometric objects in knot theory (the knot group, cyclic branched covers, etc) is a topic of immense interest that is central to modern knot theory, topology, and mathematical physics (see for example \cite{Witten1989, Witten:lectures}). Any further discussion is beyond the scope of this exposition.

\section{Band surgery obstructions via Heegaard Floer homology}
\label{sec:new}

\subsection{Bandings along the trefoil.} 
\label{subsec:trefoil} 
Motivated by our interest in site-specific recombination and by the biological significance of the class of $T(2, n)$ torus knots and links, and in particular of the trefoil knot $T(2,3)$, in a joint work with Lidman \cite{LMV} we proved the following classification theorem. Note that $T(2, 3)$ is the right-handed trefoil; an analogous statement holds for the left-handed trefoil after mirroring.

\begin{theorem}\cite[Corollary 1.2]{LMV}
\label{thm:torus}
The torus knot $T(2, n)$ is obtained from $T(2,3)$ by a non-coherent banding if and only if $n$ is $\pm 1$, 3 or 7. The torus link $T(2, n)$ is obtained from $T(2,3)$ by a coherent banding if and only if $n$ is $\pm 2$, 4 or -6.
\end{theorem}
The main effort of Theorem \ref{thm:torus} is of course to show that the trefoil is not related to $T(2, n)$ for values of $n$ other than those listed. This statement follows as a corollary of a stronger statement, Theorem \ref{thm:lens} below. Recall from Section \ref{sec:bandsurgery} that the \emph{distance} between two surgery slopes is their minimal geometric intersection number. Given a knot $K$ in a three-manifold $Y$ (e.g. $Y=L(3, 1)$), there is always a trivial filling along the meridian of $K$ for which Dehn surgery returns $Y$ \cite{Gordon}. By a \emph{distance one surgery} we mean a surgery slope that intersects the meridian of $K$ exactly once. Elsewhere these are called integral surgeries. We have:

\begin{theorem}\cite[Theorem 1.1]{LMV}
\label{thm:lens}
The lens space $L(n,1)$ is obtained by a distance one surgery along a knot in the lens space $L(3, 1)$ if and only if $n$ is one of $\pm 1, \pm 2, 3, 4, -6$ or $7$.  
\end{theorem}

Taking Theorem \ref{thm:lens} at face value, we can now prove Theorem \ref{thm:torus}.
\begin{proof}[Proof of Theorem \ref{thm:torus}]

Suppose that $n=\pm 1, \pm 2, 3, 4, -6$ or $7$. Band moves illustrating these cases are shown in Figures \ref{fig:bandings-1} and \ref{fig:bandings-2}. Conversely, suppose that there exists a non-coherent band surgery relating $T(2, 3)$ and $T(2, n)$. Using the tangle decomposition in line \eqref{tangledecomp}, we may write
\begin{equation*}
	(S^3, T(2, 3))= (B_o, t_o) \cup (B, t) \text{  and } (S^3, T(2,n))= (B_o, t_o) \cup (B, t'),
\end{equation*}
where $(B, t)$ is the zero tangle, $ (B, t')$ is the infinity tangle, and $(B_o, t_o)$ is some outside tangle. The zero and infinity tangles are the first two tangles shown in Figure \ref{fig:tangle}. 

The double cover of $S^3$ branched over $T(2, n)$ is the lens space $L(n, 1)$. We may also write $L(n, 1)$ in terms of the tangle decomposition \eqref{tangledecomp} as
\[
	L(3, 1) = \Sigma_2(B_o, t_o)\cup\Sigma_2(B, t) \text{ and } L(n, 1) = \Sigma_2(B_o, t_o)\cup\Sigma_2(B, t'),
\]
where each union of branched double covers of tangles is taken along the bounding tori. The branched double cover of the outside tangle, $\Sigma_2(B_o, t_o)$, is a compact, connected, oriented 3-manifold with torus boundary that may be alternatively described as the complement of a knot $J$ in $L(3, 1)$. This knot $J$ is the lift of a properly embedded arc arising as the core of the band from the band surgery that relates $T(2, 3)$ and $T(2,n)$. Both $\Sigma_2(S^3, T(2, 3)) = L(3, 1)$ and $\Sigma_2(S^3, T(2, n))=L(n, 1)$ are obtained by Dehn surgery along $J$. The slope yielding $L(3,1)$ is just the meridian of $J$, and the Montesinos trick \cite{Montesinos} implies that the slope yielding $L(n, 1)$ is distance one from the meridian.\footnote{A very informative exposition of the Montesinos trick can also be found in \cite[Lecture 4]{Gordon}} Theorem \ref{thm:lens} implies that $n$ cannot fall out of the set of integers $\{\pm 1, \pm 2, 3, 4, -6, 7\}$.
\end{proof}

If one is only interested in coherent band surgery, Theorem \ref{thm:torus} can be obtained more readily. The banding is coherent when $n$ is even, in which case $T(2, n)$ is a link of linking number $\pm n/2$, depending on the orientation. When the linking number is positive, the signature of $T(2, n)$ is $1-n$. The statement follows an easy consequence of Murasugi's criterion from 1965 for coherent banding:
\begin{theorem}\cite[Lemma 7.1]{Murasugi}
\label{mur}
It if two links $L$ and $L'$ are related by a coherent band surgery, then $|\sigma(L) - \sigma(L')|\leq 1$. 
\end{theorem}
When the linking number is negative, the signature of $T(2, n)$ is $1$. In this case, the classification of coherent band surgeries relating $T(2, 3)$ and $T(2, n)$ for $n$ even follows from \cite[Theorem 3.1]{DIMS}, in which they characterize coherent band surgeries relating $T(2,n)$ torus links and certain two-bridge knots.  

Our proof of Theorem \ref{thm:lens} is rather technical. In brief, the theorem is proved by analyzing the behavior of the Heegaard Floer homology $d$-invariants under distance one surgeries. The \emph{$d$-invariant} $d(Y, \mathfrak{t})$ is a rational number that comes from the Heegaard Floer homology $HF^+(Y, \mathfrak{t})$ of a rational homology sphere $Y$ equipped with a $\spinc$ structure $\mathfrak{t}$. The $d$-invariants, and the Heegaard Floer package in general, are due to~\os~\cite{OSAbs}, and analogues of them exist in other Floer theories. Each $d$-invariant $d(Y, \mathfrak{t})$ is defined to be the minimal grading of a non-torsion class in $HF^{+}(Y, \mathfrak{t})$ \cite{OSAbs}, and is in fact a $\spinc$ rational homology cobordism invariant. The most technical aspect of Theorem \ref{thm:lens} required us to prove series of formulas (see \cite{LMV}) which relate the $d$-invariants under surgery along knots in the lens space $L(3, 1)$ to certain integer-valued knot invariants that are related to Rasmussen's local $h$-invariant \cite{Rasmussenthesis}. These formulas generalize work of Ni and Wu \cite{NiWu}. 

Theorem \ref{thm:lens} can be viewed as a partial generalization of the \emph{lens space realization problem}, which was proved by Greene \cite{Greene}. The lens space realization problem asks for a characterization of the lens spaces that arise from integral surgery along nontrivial knots in $S^3$ (note that distance one surgeries are integral surgeries). This type of problem falls into a broader class of questions involving the characterization of three-manifolds arising by integral surgery along knots, and the characterization of knots that admit lens space surgeries. The Berge conjecture is one of the most prominent of these questions, having survived substantial attacks over several decades \cite{BleilerLitherland, GodaTeragaito, RasmussenLens, OSlens, Baker:SurgeryI, Baker:SurgeryII, Hedden, berge}. 
\begin{conjecture}\cite{berge}
\label{berge}
A hyperbolic knot $K$ has a lens space surgery if and only if $K$ is double primitive, and the surgery is the corresponding integral surgery.
\end{conjecture}
A hyperbolic knot is one whose complement may be given a metric of constant curvature $-1$. 
A doubly primitive knot $K$ is one that lies on the surface $F$ of a genus two Heegaard splitting $X\cup_F X'$ of $S^3$, where $[K]$ belongs to bases for the free groups $\pi_1(X)$ and $\pi_1(X')$. That the surgery coefficient is integral is due to the Cyclic Surgery Theorem \cite{CGLS}. A discussion of this conjecture can be found in \cite{Gordon}. 
Another nice discussion of some related problems in knot theory and three-manifold topology can be found in Hedden and Watson's article \cite[Section 6]{HeddenWatson}.
\begin{figure}
\centering
\includegraphics[height=0.8in]{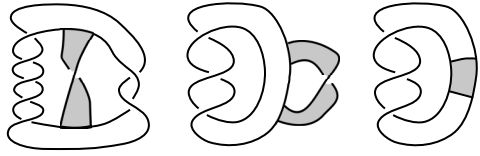}
\caption{Non-coherent bandings: (Left) $T(2,3)$ to $T(2,7)$. (Center) $T(2,3)$ to $T(2, 3)$. (Right) $T(2, 3)$ to the unknot.}
\label{fig:bandings-1}
\end{figure}

\begin{figure}
\centering
\includegraphics[height = 0.8in]{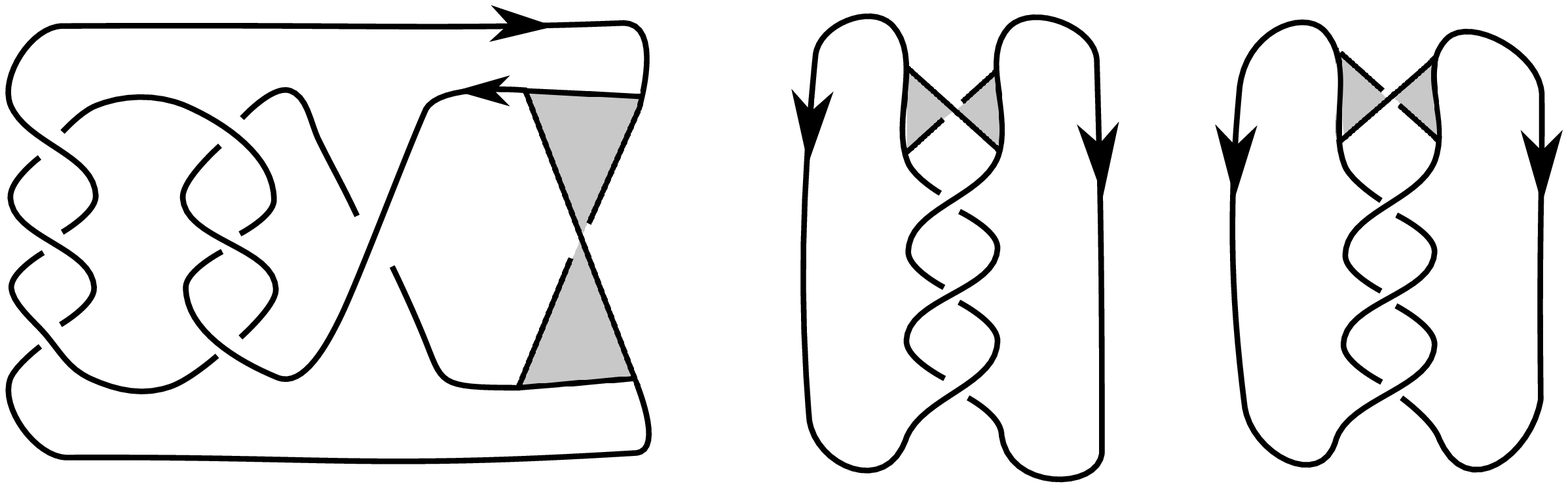}
\caption{Coherent bandings: (Left) $T(2, -6)$ to $T(2, 3)$. (Center) $T(2,3)$ to $T(2, 2)$. (Right) $T(2,3)$ to $T(2,4)$.}
\label{fig:bandings-2}
\end{figure}

\subsection{A criterion from determinant and signature.} 
Using a similar strategy as that in \cite{LMV}, that is, by studying the $d$-invariants under surgery in the branched double cover, we established a new criterion on banding via the determinant and the signature $\sigma(K)$ of a knot. Before stating our result, we need to recall the following definition of \os~\cite{OS:Branched}. Let $L_h$ and $L_v$ denote the horizontal and vertical smoothings of crossing in a fixed diagram of $L$.

\begin{definition}
\label{def:qa}
The set $QA$ of quasi-alternating links is the smallest set of links satisfying the following properties:
\begin{enumerate}
	\item The unknot is in $QA$.
	\item The set is closed under the following operation: If $L$ is a link which admits a diagram with a crossing such that 
		\begin{enumerate}
			\item both resolutions $L_h$ and $L_v$ are in $QA$
			\item $\det(L_v) \neq 0$, $\det(L_h)\neq 0$, and
			\item $\det(L) = \det(L_v) + \det(L_h)$,
		\end{enumerate}
	then $L$ is in $QA$. 
\end{enumerate}
\end{definition}
By definition, any alternating knot is quasi-alternating. 

\begin{theorem}
\label{thm:sigdif}
Let $K$ and $K'$ be a pair of quasi-alternating knots and suppose that $\det(K) = m = \det(K')$ for some square-free integer $m$. If there exists a band surgery relating $K$ and $K'$, then $|\sigma(K) - \sigma(K')|$ is $0$ or $8$. 
\end{theorem}
Our criterion for non-coherent banding (Theorem \ref{thm:sigdif}) is reminiscent of Murasugi's criterion for coherent banding (Theorem \ref{mur} in section \ref{subsec:trefoil}). It is also similar in spirit to the work of Kanenobu and Kanenobu-Moriuchi, described in Section \ref{sec:bandsurgery}, although our statement is not focused on the case of bandings relating a knot to the unknot.

We remark that given $K$ and $K'$ quasi-alternating such that $\det(K) \neq \det(K')$ (i.e. the conditions of Theorem \ref{thm:sigdif} do not hold), a banding may change the signature of a knot by arbitrary amounts. For example, take the $T(2, n)$ torus knot, which has signature $1-n$, and can be unknotted by one smoothing of a crossing. 

The following statement is a direct corollary to the proof of Theorem \ref{thm:sigdif}.
\begin{corollary}
\label{cor:lowcrossing}
Excluding $8_{19}$, let $K$ and $K'$ be knots of eight or fewer crossings with $\det(K) = m = \det(K')$ for $m$ a square-free integer. If there exists a banding from $K$ to $K'$, then $|\sigma(K) - \sigma(K')| = 8$ or $0$.
\end{corollary}
One special case where the condition that $\det(K)=\det(K')$ is met is the case of a \emph{chirally cosmetic banding}. 
\begin{definition}
A  \emph{chirally cosmetic banding} is a non-coherent band surgery relating a knot $K$ with its mirror image $K^*$. 
\end{definition}
We then asked how Theorem \ref{thm:sigdif} applies to the class of $T(2, n)$ torus knots. 
\begin{corollary}
\label{cor:torus}
The only nontrivial torus knot $T(2, n)$, with $n$ square free, admitting a chirally cosmetic banding is $T(2, 5)$.
\end{corollary}
The banding that relates $T(2, 5)$ with its mirror $T(2, -5)$ was first found by Zekovi\'c \cite{Zekovic}. Livingston has recently informed us that that the square-free condition of Corollary \ref{cor:torus} can be removed, with a possible exception when $n=9$ \cite{Livingston}. Indeed, in a recent preprint \cite{Livingston}, Livingston observes that if $K$ admits a chirally cosmetic banding, then $K \# K$ bounds a ribbon M\"obius band in the four-ball. By applying Casson-Gordon theory for non-orientable surfaces in the four-ball \cite{LivingstonGilmer}, one may then achieve an improved obstruction for the $T(2, n)$ torus knots, amongst some other knots. The status of the knot $T(2, 9)$ remains unknown.
\begin{question}
Does the torus knot $T(2, 9)$ admit a chirally cosmetic banding?
\end{question}
Another possible avenue for approaching non-coherent band surgery is to consider the bounds on nonorientable 3- and 4-genus coming from Heegaard Floer homology. 
The \emph{nonorientable 3-genus} or \emph{3-dimensional crosscap number} $\gamma_3(K)$ of $K$ in $S^3$ is the minimal first Betti number $b_1(F)$ of any surface $F$, possibly non-orientable, bounded by $K$. Likewise the \emph{nonorientable 4-genus} or \emph{smooth 4-dimensional crosscap number} $\gamma_4(K)$ is the minimal $b_1(F)$ of any surface $F$, possibly non-orientable, properly and smoothly embedded in the four-ball and bounded by $K$. Obviously $\gamma_4(K)\leq \gamma_3(K)$. The \emph{upsilon invariant} $\upsilon(K)$ is a Heegaard Floer theoretic concordance invariant defined by Ozsv\'ath, Stipsicz, and Szab\o~\cite{OSS} that gives a lower bound on the nonorientable 4-genus. 
\begin{theorem}\cite[Theorem 1.2]{OSS}
Let $K$ be a knot in $S^3$. Then 
\[
	|\upsilon(K) - \frac{\sigma(K)}{2}|\leq \gamma_4(K).
\]
\end{theorem}
Livingston's observation, when combined with the lower bound from $\upsilon$, implies the following.
\begin{theorem}\cite[Theorem 19]{Livingston}
If $K$ admits a chirally cosmetic banding, then $\upsilon(K)=\sigma(2)/2$.
\end{theorem}
We note that for alternating and quasi-alternating knots, it is always true that $\upsilon(K)=\sigma(2)/2$. However, for non-quasi-alternating knots, this provides a good obstruction to chirally cosmetic bandings. The literature on nonorientable 3- and 4-genus is relatively sparse, compared with its orientable counterpart. Additional work in this direction (particularly in dimension four) can be found in \cite{Batson, JabukaKelly, LivingstonGilmer, OSS, Yasuhara}.   

\section{The frequency of chirally cosmetic bandings, as assessed by numerical simulations}
\label{sec:numerical}
The primary outcome of our analytical work in Theorem \ref{thm:torus}, Theorem \ref{thm:sigdif}, and their corollaries is to establish mathematical obstructions to the existence of a banding relating a given pair of knots. That is to say, we aim to determine whether a banding cannot possibly occur, and if possible, find examples of bands when such obstructions are not available. We complement the analytical work with numerical simulations of non-coherent banding on cubic lattice knots for which the focus is to determine whether a band surgery relating two knots exists, and if so, determine the relative frequency of such a transition. This study is an ongoing effort. We will discuss here the preliminary findings pertaining to chirally cosmetic bandings, as reported in \cite{MV}.

\subsection{Cubic lattice knots.} A \emph{cubic lattice knot} is an embedding of a knot in the simple cubic lattice, $(\R\times \Z\times \Z) \cup (\Z\times\R\times \Z)\cup (\Z\times\Z\times\R)$, with vertices contained on the integer lattice $\Z^3$. Any tame knot $K$ admits a cubic lattice representative, i.e. a self-avoiding polygon (SAP) with knot type $K$, called a \emph{lattice conformation} of $K$. The length of a fixed conformation is the number of unit steps in the embedding. 

In computational knot theory, particularly in numerical studies of DNA topology, it is common practice to model knots and links as random SAPs in $\mathbb{R}^3$ (e.g. freely jointed chains, uniform random polygons, wormlike chains) or in $\Z^3$ (i.e. lattice knots and links). 
While polymer models in $\mathbb{R}^3$ better approximate the physical properties of DNA, modeling DNA as lattice polygons presents a number of analytical advantages. 
There is extensive analytical work on knotting and linking in $\Z^3$, and a solid mathematical foundation for the development of algorithms that generate equilibrium distributions of SAPs in $\Z^3$ \cite{MadrasSlade, WangLandau}. Also, lattice polygons incorporate inherently a volume exclusion, which can be used to model the effective diameter of DNA. More accurate modeling of DNA as a SAP in $\Z^3$ is achieved by incorporating excluded volume effects. In collaboration with C. Soteros' group at the University of Saskatchewan \cite{Schmirler}, we have work in progress in this direction based on the work of \cite{Tesi}.

\subsection{Chirally cosmetic bandings are rare.} Our analytical work suggests that chirally cosmetic bandings are unusual events. This is consistent with the behavior observed in our simulations of non-coherent banding in the cubic lattice. For each chiral knot type $K$ of up to eight crossings, we sampled three million reconnection events and produced a transition probability network to summarize the outcome. Each reconnection represented a non-coherent banding. The implementation was adapted from our published work in \cite{Stolz2017} (see Section \ref{methods} for more details). 

In our preliminary report, we found chirally cosmetic bandings to be exceedingly rare. Chirally cosmetic bandings were observed, with transition probabilities given in Table \ref{table:chiral}, for only three knots: $5_1$, $8_8$, and $8_{20}$. The band move relating $5_1$ and its mirror image was known to exist, and had been previously reported in \cite{Zekovic}. The band move for $8_8$ had not, to our knowledge, been explicitly reported in the literature but could be easily ascertained from the symmetric union presentation given in \cite{Lamm}. The banding exhibited for $8_8^*$ was newly discovered via our computer simulations. 

\begin{table}[t]

\caption{
\footnotesize{\textbf{Chirally cosmetic bandings are rare.} The relative likelihood of chirally cosmetic bandings for knots of up to eight crossings is given. The probabilities were computed out of $3\times 10^6$ total band moves performed on each knot type.
}
}
\begin{center}
  \begin{tabular}{ r  r  c  r  c  }
    \hline
    \emph{Knot} & $P(K-K^*) \times 10^{-5}$  &  \emph{Number observed}   \\
    \hline
    $5_1$ & 3.467 &  104  \\ 
    $5_1^*$ & 2.800 &   84   \\ 
    $8_8$ & 0 & 0  \\ 
    $8_8^*$ & 0.003 &  1  \\ 
    $8_{20}$ & 42.833  & 1285  \\ 
    $8_{20}^*$ & 46.400 &  1392 \\ 
    \hline
  \end{tabular}
\end{center}

\label{table:chiral}
\end{table}

\subsection{Methods.}
\label{methods}

In the experiment that we designed, we considered the set of 63 non-trivial prime knots of up to eight crossings, where we included both knots from each chiral pair. For each substrate knot $K$, we used the BFACF algorithm to generate large ensembles of cubic lattice representatives of $K$. The BFACF algorithm is a Markov chain Monte Carlo algorithm frequently used in the study of self-avoiding polygons \cite{MadrasSlade}. It is ergodic in the state space of all cubic lattice representatives of a self-avoiding polygon. The algorithm acts by perturbing the chain via a set of small moves, called the BFACF moves. These moves preserve the isotopy type of the knotted chain, but may change the length. Our particular implementation of BFACF employed a composite Markov chain (CMC) Monte Carlo process in order to more effectively randomize samples.

\begin{figure}
\includegraphics[width=0.6\textwidth]{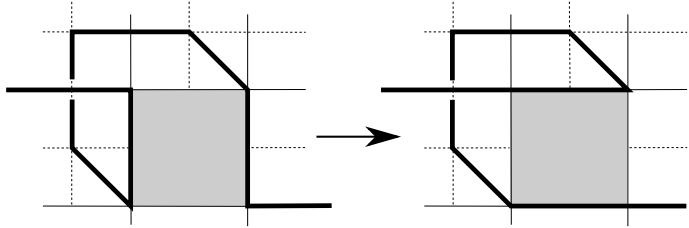}
\label{fig:reconnectionsites}
\caption{An example of reconnection sites in the cubic lattice.}
\end{figure}

Having generated large, independent ensembles of cubic lattice representatives of a fixed knot type, we then searched for reconnection sites, performed recombination, and identified the resulting product knots. In the cubic lattice, a reconnection site is a pair of edges along the chain that are of a unit distance from each other and occupy the opposite edges of a square in the lattice. See Figure \ref{fig:reconnectionsites}. This was accomplished by extending our software package, originally developed for the case of coherent band surgery simulations in \cite{Stolz2017}, to the non-coherent case. We also similarly use batch mean analysis and ratio estimation techniques to ensure statistical robustness. For knot identification, we primarily used the HOMFLY-PT polynomial, with a second identification routine applied in the few cases when identification via the HOMFLY-PT polynomial was inconclusive. Extended information on the BFACF algorithm and our implementation can be found in \cite{MV}, \cite[Supplementary Materials]{Stolz2017}, \cite{MadrasSlade}.

\subsection{Future studies.}
The difference in chirality between a chiral knot $K$ and its mirror image $K^*$ is a well-defined concept. Thus, in \cite{MV}, we focused on bandings that relate a chiral knot $K$ with $K^*$. For an arbitrary pair of knots $K$ and $K'$, the relative difference in chirality is a quantity that must formally be defined and estimated by numerical simulations. A more extensive study of the chirality trends apparent in the transition probability networks associated to non-coherent banding and its biological relevance is the focus of a forthcoming report \cite{FMV}.  

\subsection*{Acknowledgements} We are especially grateful to our collaborator Tye Lidman. We also thank Michelle Flanner for her sustained efforts in carrying out the numerical studies described here and Chuck Livingston for sharing with us his new work on chirally cosmetic bandings.

\bibliographystyle{amsplain}
\bibliography{biblio}

\end{document}